\documentclass{amsart}

 \usepackage{amssymb,latexsym,amsmath,amsthm}

\usepackage[numbers,sort&compress]{natbib}

 \renewcommand{\div}{\mathop{\mathrm{div}}\nolimits}

\newtheorem*{thm*}{Theorem A}

\newtheorem{thm}{Theorem}[section]

\newtheorem{dfn}{Definition}[section]

\newtheorem{lemma}{Lemma}[section]

\numberwithin{equation}{section}

\begin{document}

\def\IR{{\mathbb{R}}}

\title{Stable solutions of symmetric systems on Riemannian manifolds}

\author{Mostafa Fazly}

\address{Department of Mathematics, The University of Texas at San Antonio, San Antonio, TX 78249, USA}
\email{mostafa.fazly@utsa.edu}

\address{Department of Mathematics \& Computer Science, University of Lethbridge, Lethbridge, AB T1K 3M4 Canada.}
\email{mostafa.fazly@uleth.ca}

\thanks{The author gratefully acknowledges Natural Sciences and Engineering Research Council of Canada (NSERC) Discovery Grant and University of Texas at San Antonio Start-up Grant.}

\maketitle

\begin{abstract}  
We examine stable solutions of the following symmetric system on a complete, connected, smooth Riemannian
manifold $\mathbb{M}$ without boundary,  
\begin{equation*}
   -\Delta_g u_i   =   H_i(u_1,\cdots,u_m)  \ \ \text{on} \ \ \mathbb{M},
  \end{equation*}  
when $\Delta_g$ stands for the Laplace-Beltrami operator, $u_i:\mathbb{M}\to \mathbb R$ and $H_i\in C^1(\mathbb R^m) $ for $1\le i\le m$.    This system is called symmetric if the matrix of partial derivatives of all components of $H$, that is $\mathbb H(u)=(\partial_j H_i(u))_{i,j=1}^m$, is symmetric. We prove  a stability inequality and a Poincar\'{e} type inequality for stable solutions using the Bochner-Weitzenb\"{o}ck formula. Then, we apply these inequalities to establish  Liouville theorems and  flatness
of level sets for stable solutions of the above symmetric system, under certain assumptions on the manifold and on solutions.   
\end{abstract}


\section{Introduction} 
Suppose that  $\mathbb M$ is  a complete, connected, smooth, $n$-dimensional Riemannian
manifold without boundary, endowed with a smooth Riemannian metric $g = \{g_{ij}\}$.  Consider $u=(u_i)_{i=1}^m$ for $u_i\in C^3(\mathbb M)$  that satisfies 
 \begin{equation} \label{main}
   -\Delta_g u_i   =   H_i(u)  \quad  \text{on} \ \  \mathbb{M} , 
  \end{equation}  
 where $\Delta_g$ stands for the Laplace-Beltrami operator and $H_i\in C^1(\mathbb R^m)$ for  $i=1,\dots, m$. Let $f:\mathbb M\to \mathbb R$ be a  function in $C^3(\mathbb M)$. Then,  the Riemannian gradient and the Laplace-Beltrami operator are given by 
\begin{equation}\label{}
(\nabla_g f)_i=g^{ij} \partial_j f,
\end{equation}
and 
\begin{equation}\label{}
\Delta_g f= \div_g(\nabla_g f)=\frac{1}{\sqrt{|g|}} \partial_i\left( \sqrt{|g|} g^{ij} \partial_j f\right). 
\end{equation}
In this article, we frequently refer to the  Bochner-Weitzenb\"{o}ck formula that is 
\begin{equation}\label{bochner}
\frac{1}{2} \Delta_g |\nabla _g f|^2 = |\mathcal H _f|^2+ \nabla _g \Delta_g f\cdot \nabla_g f+ \text{Ric}_g (\nabla_g f,\nabla_g f). 
\end{equation}
 Here $\mathcal H_f$ stands for the Hessian
of $f$ as the symmetric 2-tensor given in a local patch that is 
\begin{equation}\label{Hu}
(\mathcal H_f)_{ij}:=\partial_{ij} f- \Gamma_{ij}^k \partial_k f,
\end{equation}
where $\Gamma_{ij}^k$ is the Christoffel symbol that is 
\begin{equation}\label{gijk}
 \Gamma_{ij}^k= \frac{1}{2} g^{hk}(\partial_i g_{hj}+\partial_j g_{ih}-\partial_{h} g_{ij}). 
\end{equation}
Note that from the definition of the Hessian one can see that 
\begin{equation}\label{ineqfH}
|\nabla_g |\nabla_g f| |^2 \le |\mathcal H_f|^2. 
\end{equation}
The equality holds at $p\in \mathbb M \cap \{\nabla_g f \neq 0\}$ if and only if there exists $\kappa_k:\mathbb M\to \mathbb R$ for each $k=1,\cdots,n$ such that 
\begin{equation}\label{kappa}
\nabla_g(\nabla_g f)_k(p)=\kappa_k (p) \nabla_g f(p),
\end{equation}
see \cite{fsv,ll,jo} for details.   Here,  we provide the definition of parabolic manifolds. We refer interested readers to \cite{gt,ls,r,cy} for more information.  
\begin{dfn} \label{paraman}
A manifold  $\mathbb M$ is called parabolic when for every point $p \in \mathbb M$ there exists a precompact neighborhood  $M_p$ of $p$ in $\mathbb M$ such that for an arbitrary positive $\epsilon$ there exists a function   $f_\epsilon \in C_c^\infty(\mathbb M)$ that $f_\epsilon(q)=1$ for all $q\in M_p$ and $\int_{\mathbb M}|\nabla _g f_\epsilon|^2 dV_g \le \epsilon $.  
\end{dfn} 
For the sake of simplicity,  we use the notation  $\partial_j H_i(u)=\frac{\partial H_i(u)}{\partial {u_j}} $ and  we  assume that 
\begin{equation}\label{partialiHj}
\partial_i H_j (u)\partial_j H_i (u) > 0, 
\end{equation} 
for $1\le i,j\le m$.    In addition, $B_R$ stands for a geodesic ball of radius $R>0$ centred at a 
given point of  $\mathbb M$ and  $|B_R|$ is its volume with respect to the volume element
$dV_g$.   We now provide the notion of stable solutions.  
\begin{dfn} \label{stable}
A solution $u=(u_i)_{i}$ of (\ref{main}) is said to be stable when there exist a sequence of functions $\zeta=(\zeta_i)_i $ where each $\zeta_i\in C^3(\mathbb M)$ does not change sign and a nonnegative  constant $\lambda$ such that the following linearized system holds
\begin{equation} \label{sta}
 -\Delta \zeta_i =   \sum_{j=1}^{n} \partial_j H_i(u) \zeta_j + \lambda \zeta_i \quad  \text{on}  \ \  \mathbb M,
 \end{equation} 
 for all $ i=1,\cdots,m$.  In addition, we assume that $\partial_j H_i(u) \zeta_i \zeta_j >0$ for all $1\le i,j\le m$. 
\end{dfn} 
  For the case of $\mathbb M=\mathbb R^n$ and scalar equations,  that is when $m=1$,  the above notion of stability was derived in connections with the De Giorgi's conjecture and it is used in the literature extensively,  see \cite{aac,ac,dkw,df,gg1,sav,fsv1,fsv2,fsv,pw}. The latter conjecture  \cite{deg}, given in 1978, states that bounded monotone solutions of the Allen-Cahn equation must be hyperplane, see \cite{aac,ac,fsv,gg1,pw} and references therein.  For the case of system of equations, that is when $m\ge 2$
  \begin{equation}
   -\Delta u_i   =   H_i(u_1,\cdots,u_m)  \ \ \text{on} \ \ \mathbb R^n,
  \end{equation} 
   the notion of stability is given in \cite{mf,mf1,fg} and references therein.  We also refer interested readers to \cite{blwz,btww,fso,wang} for the following two-component elliptic system,  originated in the Bose-Einstein condensation and nonlinear optics, 
   \begin{eqnarray}\label{Bose}
 \left\{ \begin{array}{lcl}
\hfill -\Delta u =H_1(u,v) \quad  \text{on}  \ \  \mathbb R^n, \\
\hfill  -\Delta v = H_2(u,v) \quad  \text{on}  \ \  \mathbb R^n,
\end{array}\right.
  \end{eqnarray}
when $H_1(u,v)=-uv^2$ and $H_2(u,v)=-vu^2$.     For this system,  nonnegative monotone solutions, which are $u_{x_n}(x) v_{x_n}(x)<0$ in $\mathbb R^n$,  are of interests. Straightforward computations show that monotone solutions satisfy  (\ref{sta}) for $\zeta_1=u_{x_n}$, $\zeta_2=v_{x_n}$ and $\lambda=0$. Note also that  $\partial_v H_1=\partial_u H_2<0$ and $\partial_v H_1\partial_u H_2>0$. For a similar notion of stability,   we refer interested readers to \cite{fg,abg} for the Allen-Cahn system and to \cite{Mont,d,dp,cf,mf1,mf} for systems with general nonlinearities on bounded and unbounded domains.

We now provide the notion of symmetric systems introduced in \cite{mf} when $\mathbb M=\mathbb R^n$.  Symmetric systems play a fundamental role throughout this paper when we study system (\ref{main}) with a general nonlinearity $H(u)=(H_i(u))_{i=1}^m$.    
\begin{dfn}\label{symmetric} We call system (\ref{main}) symmetric if the matrix of gradient of all components of $H$ that is 
 \begin{equation} \label{H}
\mathbb{H}:=(\partial_i H_j(u))_{i,j=1}^{m} ,
 \end{equation}
  is symmetric. 
   \end{dfn}

Here is how this article is structured.  In Section \ref{secin}, we provide a stability inequality for stable solutions of (\ref{main}). Applying this stability inequality we prove a weighted Poincar\'{e} type inequality for stable solutions of system (\ref{main}).    In Section \ref{secli}, we provide applications of  inequalities provided in Section \ref{secin}.     To be mathematically more precise, we establish various Liouville theorems regarding stable solutions as well as a rigidity result concerning level sets of stable solutions. 
   
\section{Inequalities for stable solutions} \label{secin}

We start this section by an inequality for stable solutions of (\ref{main}). For the case of a scalar equations,  this inequality  is used in the literature extensively to study  symmetry  properties, regularity theory, Liouville theorems, etc.  regarding stable solutions. We refer interested readers to \cite{aac,fcs,ac,gg1,fsv,fsv1,fsv2,ng,c} for more information.    For the case of system of equations and when $\mathbb M=\mathbb R^n$, this inequality is given in \cite{cf,fg,mf}.

\begin{lemma}\label{stablein}  
  Let $u$ denote a stable solution of (\ref{main}).  Then 
\begin{equation} \label{stability}
\sum_{i,j=1}^{m} \int_{\mathbb M}  \sqrt{\partial_j H_i(u) \partial_i H_j(u)} \phi_i \phi_j  dV_g\le \sum_{i=1}^{m} \int_{\mathbb M}  |\nabla_g \phi_i|^2 dV_g  , 
\end{equation} 
for any $\phi=(\phi_i)_{i=1}^m$ where $ \phi_i \in C_c^1(\mathbb M)$ for $1\le i\le m$. 
\end{lemma}  

\begin{proof}   The fact that $u$ is a stable solution implies that  there exist a  sequence $\zeta=(\zeta_i)_{i=1}^m$ and a nonnegative constant $\lambda$ such that for all $i=1,\cdots, m$  
\begin{equation}\label{linearzeta}
  -\Delta_g \zeta_i =  \sum_{j=1}^{n} \partial_j H_i(u) \zeta_j  + \lambda \zeta_i \ \ \text{on } \ \mathbb M  . 
\end{equation}
Consider a sequence of test functions  $\phi=(\phi_i)_i^m$ where $ \phi_i \in C_c^1(\mathbb M)$ for $1\le i\le m$.  Multiplying  both sides of (\ref{linearzeta})  with $\frac{\phi_i^2}{\zeta_i}$ and integrating,  we get 
\begin{equation}\label{ineq}
 \lambda \int_{\mathbb M}  \phi^2_i + \sum_{j=1}^{n}  \int_{\mathbb M}  \partial_j H_i(u) \zeta_j \frac{\phi_i^2}{\zeta_i}dV_g   = \int_{\mathbb M}    - \frac{\Delta_g \zeta_i}{\zeta_i} {\phi_i^2}dV_g . 
  \end{equation}
Applying the fact that 
\begin{equation} 2 \frac{\phi_i}{\zeta_i}\nabla_g \zeta_i\cdot \nabla_g \phi_i - |\nabla_g \zeta_i|^2 \frac{\phi_i^2}{\zeta_i^2} \le |\nabla_g \phi_i |^2, 
  \end{equation}
we  obtain    
\begin{equation}\label{deltagzeta}
\int_{\mathbb M}    - \frac{\Delta_g \zeta_i}{\zeta_i} {\phi_i^2} dV_g  \le \int_{\mathbb M}  |\nabla_g \phi_i|^2  dV_g, 
  \end{equation}
   for each $i=1,\cdots,m$. Note also that here we have applied the divergence theorem for the Laplace-Beltrami operator. Combining (\ref{deltagzeta}) and (\ref{ineq}), we end up with 
   \begin{equation}\label{jhizeta}
 \sum_{i,j=1}^{n}  \int_{\mathbb M}  \partial_j H_i(u) \zeta_j \frac{\phi_i^2}{\zeta_i}dV_g   \le  \sum_{i=1}^{n} \int_{\mathbb M}  |\nabla_g \phi_i|^2  dV_g.  
  \end{equation}
   For the left-hand side of (\ref{jhizeta}), straightforward calculations show that  
\begin{eqnarray}
 \nonumber \sum_{i,j=1}^{m}  \int_{\mathbb M}  \partial_j H_i(u) \zeta_j \frac{\phi_i^2}{\zeta_i}dV_g & =& \nonumber  \sum_{i<j}^{m}  \int_{\mathbb M}  \partial_j H_i(u) \zeta_j \frac{\phi_i^2}{\zeta_i} dV_g + \sum_{i>j}^{n}  \int_{\mathbb M}  \partial_j H_i(u) \zeta_j \frac{\phi_i^2}{\zeta_i} dV_g  
  \\&& \nonumber +  \sum_{i=1}^{m}\int_{\mathbb M} \partial_i H_i(u) {\phi_i^2}dV_g 
  \\&=& \nonumber
  \sum_{i<j}^{m}   \int_{\mathbb M} \partial_j H_i(u) \zeta_j \frac{\phi_i^2}{\zeta_i} dV_g  + \sum_{i<j}^{m}  \int_{\mathbb M}  \partial_i H_j(u) \zeta_i \frac{\phi_j^2}{\zeta_j} dV_g  
  \\&&\nonumber  + \sum_{i=1}^{m} \int_{\mathbb M} \partial_i H_i(u) {\phi_i^2}dV_g 
  \\&=& \nonumber \sum_{i<j}^{m}   \int_{\mathbb M}  \left( \partial_j H_i(u) \zeta_j \frac{\phi_i^2}{\zeta_i}   + \partial_i H_j(u) \zeta_i \frac{\phi_j^2}{\zeta_j}  \right)dV_g  
  \\&& \nonumber  +  \sum_{i=1}^{m} \int_{\mathbb M}  \partial_i H_i(u) {\phi_i^2} dV_g 
  \\&\ge &\nonumber  2 \sum_{i<j}^{m}  \int_{\mathbb M}    \sqrt{\partial_j H_i(u) \partial_i H_j(u)} \phi_i \phi_j dV_g +  \sum_{i=1}^{m} \int_{\mathbb M}  \partial_i H_i(u) {\phi_i^2} dV_g 
  \\&=&\sum_{i,j=1}^{m}   \int_{\mathbb M} \sqrt{\partial_j H_i(u) \partial_i H_j(u)} \phi_i\phi_j dV_g .
  \end{eqnarray}
This finishes the proof. 

\end{proof}

Applying the stability inequality (\ref{stability}), given in Lemma \ref{stablein},  we prove a weighted Poincar\'{e} inequality as the following theorem. Note that for the case of scalar equations a similar  inequality   is established in \cite{sz}  and used in \cite{sz1,fsv,c},  in the Euclidean sense, and  in \cite{fsv1,fsv2} on Riemannian manifolds.   For the case of system of equations, that is when $m\ge 2$, this inequality was derived in \cite{fg} in the Euclidean sense and  later used in \cite{dp,d,mf1} and references therein.   

\begin{thm}\label{poin}
 Assume that   $u \in C^3(\mathbb M)$ is a stable solution of (\ref{main}) where $m,n\ge 1$.  Then, the following inequality holds for any $\eta=(\eta_k)_{k=1}^m \in C_c^1(\mathbb M)$ 
 \begin{eqnarray}\label{poincare}
&&  \sum_{i=1}^m   \int_{\mathbb M}  \left(  Ric_g(\nabla_g u_i,\nabla_g u_i) +|\mathcal H_{u_i}|^2-|\nabla _g |\nabla_g u_i| |^2  \right)\eta_i^2 dV_g 
\\&&\nonumber +\sum_{i\neq j}^m  \int_{\mathbb M} \left(\sqrt{\partial_j H_i(u) \partial_i H_j(u)  }   |\nabla_g u_i| |\nabla_g u_j| \eta_i \eta_j-  \partial_j H_i(u) \nabla_g u_i \cdot\nabla_g  u_j \eta_i^2 \right) dV_g 
\\&\le& \nonumber   \sum_{i=1}^m   \int_{\mathbb M}  |\nabla_g u_i|^2   |\nabla_g \eta_i|^2dV_g . 
  \end{eqnarray} 
\end{thm}
\begin{proof} Test the stability inequality (\ref{stability}) on $\phi_i=|\nabla_g u_i|\eta_i$ for the sequence of test functions $\eta=(\eta_i)_{i=1}^m$ and $\eta_i\in C^1_c(\mathbb M)$,  to get 
 \begin{eqnarray}\label{stablepoin}
  I &:=& \sum_{i,j=1}^m  \int_{\mathbb M} \sqrt{\partial_j H_i(u) \partial_i H_j(u)  }   |\nabla_g u_i| |\nabla_g u_j| \eta_i \eta_j dV_g  
\\&\le& \nonumber \sum_{i=1}^m \int_{\mathbb M}  |\nabla_g (|\nabla_g u_i|\eta_i)|^2dV_g  =:J  . 
   \end{eqnarray} 
We rewrite $I$ as 
\begin{eqnarray}\label{IJ}
 I &=&\sum_{i=1}^m  \int_{\mathbb M}  |\partial_i H_i(u)|   |\nabla_g u_i|^2  \eta^2_i dV_g  
\\&& \nonumber + \sum_{i\neq j}^m  \int_{\mathbb M} \sqrt{\partial_j H_i(u) \partial_i H_j(u)  }   |\nabla_g u_i| |\nabla_g u_j| \eta_i \eta_j dV_g .
\end{eqnarray}
A simple integration by parts implies  
  \begin{eqnarray}\label{J}
&& J= \sum_{i=1}^m \int_{\mathbb M} |\nabla_g |\nabla_g u_i| |^2 \eta_i^2 + |\nabla_g u_i|^2 |\nabla_g\eta_i|^2 +2 (\eta_i |\nabla_g u_i|) \nabla_g |\nabla_g u_i|\cdot \nabla_g\eta_i  dV_g 
\\&=&\nonumber \sum_{i=1}^m \int_{\mathbb M} |\nabla_g |\nabla_g u_i| |^2 \eta_i^2 + |\nabla_g u_i|^2 |\nabla_g\eta_i|^2 -\frac{1}{2} \eta^2_i  \Delta_g |\nabla_g u_i|^2 dV_g .
    \end{eqnarray} 
We now apply the Bochner-Weitzenb\"{o}ck formula, given in (\ref{bochner}),  to have 
\begin{eqnarray}\label{J2}
&& J= \sum_{i=1}^m \int_{\mathbb M} |\nabla_g |\nabla_g u_i| |^2 \eta_i^2 + |\nabla_g u_i|^2 |\nabla_g\eta_i|^2  dV_g 
\\&&\nonumber- \sum_{i=1}^m \int_{\mathbb M}  |\mathcal H _{u_i}|^2 \eta^2_i +  \nabla _g \Delta_g u_i \cdot \nabla_g u_i \eta^2_i + \text{Ric}_g (\nabla_g u_i,\nabla_g u_i) \eta^2_i dV_g .
     \end{eqnarray} 
Combining  (\ref{J2}), (\ref{IJ}) and (\ref{stablepoin}) implies 
  \begin{eqnarray}\label{Hij1}
&&\sum_{i=1}^m  \int_{\mathbb M}  |\partial_i H_i(u)|   |\nabla_g u_i|^2  \eta^2_i dV_g 
 \\& \le & \nonumber-\sum_{i\neq j}^m  \int_{\mathbb M} \sqrt{\partial_j H_i(u) \partial_i H_j(u)  }   |\nabla_g u_i| |\nabla_g u_j| \eta_i \eta_j dV_g 
 \\&&\nonumber +\sum_{i=1}^m \int_{\mathbb M} |\nabla_g |\nabla_g u_i| |^2 \eta_i^2 + |\nabla_g u_i|^2 |\nabla_g\eta_i|^2  dV_g 
\\&&\nonumber- \sum_{i=1}^m \int_{\mathbb M}  |\mathcal H _{u_i}|^2 \eta^2_i +  \nabla _g \Delta_g u_i \cdot \nabla_g u_i \eta^2_i + \text{Ric}_g (\nabla_g u_i,\nabla_g u_i) \eta^2_i dV_g .
      \end{eqnarray} 
On the other hand, for each $i=1,2,...,m$ differentiating both sides  of the $i^{th}$ equation in (\ref{main}) and multiplying  with $\nabla_g u_i \eta_i^2$ we get 
\begin{equation}\label{}
-\nabla_g u_i \cdot  \nabla_g \Delta_g u_i \eta_i^2 = \sum_{j=1}^m \partial_j H_i(u) \nabla_g u_j \cdot \nabla_g u_i \eta_i^2. 
\end{equation}
Integrating both sides of the above equation gives  
\begin{eqnarray}\label{KL}
I_1 &:=&-\sum_{i=1}^m\int_{\mathbb M}\nabla_g u_i \cdot  \nabla_g \Delta_g u_i \eta_i^2 dV_g  
\\&=&\nonumber  \sum_{i,j=1}^m \int_{\mathbb M}  \partial_j H_i(u) \nabla_g u_j \cdot \nabla_g u_i \eta_i^2dV_g =:J_1.
\end{eqnarray}
Note that $J_1$ can be rewritten as 
\begin{equation}\label{L}
J_1=  \sum_{i=1}^m \int_{\mathbb M}  \partial_i H_i(u) |\nabla_g u_i |^2 \eta_i^2 dV_g +\sum_{i\neq j}^m  \int_{\mathbb M}  \partial_j H_i(u) \nabla_g u_j \cdot \nabla_g u_i  \eta_i^2 dV_g . 
\end{equation}
From (\ref{L}) and (\ref{KL}) we get 
  \begin{eqnarray}\label{Hij2}
\sum_{i=1}^m \int_{\mathbb M}  \partial_i H_i(u) |\nabla_g u_i |^2 \eta_i^2 dV_g 
 &=& -\sum_{i\neq j}^m \int_{\mathbb M}  \partial_j H_i(u) \nabla_g u_j \cdot \nabla_g u_i  \eta_i^2dV_g 
 \\&& \nonumber - \sum_{i=1}^m \int_{\mathbb M}\nabla_g u_i \cdot  \nabla_g \Delta_g u_i \eta_i^2dV_g.  
      \end{eqnarray} 
Combining (\ref{Hij1}) and (\ref{Hij2}),  completes the proof. 
     
\end{proof}

\section{Liouville theorems}\label{secli}

In this section, we provide various Liouville theorems for solutions of (\ref{main})  as consequences of the Poincar\'{e} inequality given in Theorem \ref{poin}.  We start with the following Liouville theorem. 

\begin{thm}\label{Mcom} 
Suppose that $u=(u_i)_{i=1}^m$ is a stable solution of symmetric system (\ref{main}) where the Ricci curvature is nonnegative and $Ric_g$ is not identically zero.  Assume also that one of the following conditions holds  
\begin{enumerate}
\item[(i)] $\mathbb M$ is compact.
\item[(ii)]  $\mathbb M$ is complete and parabolic and $|\nabla_g u_i|\in L^\infty(\mathbb M)$ for each $i=1,\cdots,m$.
\end{enumerate}
Then, each $u_i$ must be constant for $i=1,\cdots,m$.
\end{thm}
\begin{proof} We start with Part (i).   If $\mathbb M$ is compact, then we set $\eta_i=1$ in the stability inequality to get 
 \begin{eqnarray}\label{poincare1}
&&  \sum_{i=1}^m   \int_{\mathbb M}  \left(  Ric_g(\nabla_g u_i,\nabla_g u_i) +|\mathcal H_{u_i}|^2-|\nabla _g |\nabla_g u_i| |^2  \right) dV_g 
\\&&\nonumber +\sum_{i\neq j} \int_{\mathbb M} \left(|\partial_j H_i(u)|   |\nabla_g u_i| |\nabla_g u_j| -  \partial_j H_i(u) \nabla_g u_i \cdot\nabla_g  u_j  \right) dV_g \le0.
  \end{eqnarray} 
Note that 
\begin{equation}\label{jHi}
|\partial_j H_i(u)|   |\nabla_g u_i| |\nabla_g u_j| \ge  \partial_j H_i(u) \nabla_g u_i \cdot\nabla_g  u_j.
\end{equation}
 This and (\ref{poincare1}) imply that 
\begin{equation}\label{ricM}
\sum_{i=1}^m   \int_{\mathbb M}  \left[  Ric_g(\nabla_g u_i,\nabla_g u_i) +|\mathcal H_{u_i}|^2-|\nabla _g |\nabla_g u_i| |^2  \right] dV_g   \le 0 .
\end{equation}
Note that the Ricci curvature is nonnegative and the following inequality holds %applying (\ref{ineqfH}) that is 
\begin{equation}\label{inui}
|\nabla_g |\nabla_g u_i| |^2 \le |\mathcal H_{u_i}|^2 . 
 \end{equation} 
From this and (\ref{ricM}),  we get 
\begin{equation}\label{ricMH}
 \int_{\mathbb M}  Ric_g(\nabla_g u_i,\nabla_g u_i) dV_g =\int_{\mathbb M} \left[ |\mathcal H_{u_i}|^2-|\nabla _g |\nabla_g u_i| |^2 \right]dV_g  =0.
 \end{equation}
This implies that for every $p\in\mathbb M\cap\{\nabla_g u_i\neq0\}$,  we have
\begin{equation}\label{inuip}
|\nabla_g |\nabla_g u_i| |^2 (p)= |\mathcal H_{u_i}|^2(p), 
 \end{equation} 
 and 
\begin{equation}\label{ricp}
 Ric_g(\nabla_g u_i,\nabla_g u_i) (p)=0. 
 \end{equation} 
Note that (\ref{inuip}) implies that 
\begin{equation}\label{kappa}
\nabla_g(\nabla_g u_i)_k(p)=\kappa_k (p) \nabla_g u_i(p). 
\end{equation}
 From the assumptions on $Ric_g$, we conclude that  $Ric_g$ is positive definite in some open subset of $\mathbb M$. This and (\ref{ricp}) imply that for each $i$, we have $\nabla_g u_i(p)=0$ for any $p$ in an open subset of $\mathbb M$.  Therefore, the unique continuation principle \cite{kaz} implies that each $u_i$ must be constant on $\mathbb M$. 

We now consider Part (ii). Since $\mathbb M$ is parabolic, from Definition \ref{paraman} there exist $M_p$ and $f_\epsilon$ such that $\int_{\mathbb M}| \nabla_g f_\epsilon|^2 dV_g \le \epsilon $.  Set $\eta_i :=f_\epsilon$ in the Poincar\'{e} inequality (\ref{poincare}) to get 
 \begin{eqnarray}\label{poincare2}
&&  \sum_{i=1}^m   \int_{M_p}  \left(  Ric_g(\nabla_g u_i,\nabla_g u_i) +|\mathcal H_{u_i}|^2-|\nabla _g |\nabla_g u_i| |^2  \right) dV_g 
\\&&\nonumber +\sum_{i\neq j} \int_{M_p} \left(|\partial_j H_i(u)|   |\nabla_g u_i| |\nabla_g u_j| -  \partial_j H_i(u) \nabla_g u_i \cdot\nabla_g  u_j  \right) dV_g 
\\&\le&\nonumber \sum_{i=1}^m   \int_{\mathbb M}  \left(  Ric_g(\nabla_g u_i,\nabla_g u_i) +|\mathcal H_{u_i}|^2-|\nabla _g |\nabla_g u_i| |^2  \right) f_\epsilon  dV_g 
\\&&\nonumber +\sum_{i\neq j} \int_{\mathbb M} \left(|\partial_j H_i(u)|   |\nabla_g u_i| |\nabla_g u_j| -  \partial_j H_i(u) \nabla_g u_i \cdot\nabla_g  u_j  \right)f^2_\epsilon  dV_g 
\\&\le&\nonumber \sum_{i=1}^m   \int_{\mathbb M}  |\nabla_g u_i|^2   |\nabla_g f_\epsilon|^2 dV_g 
\\&\le& \nonumber \sum_{i=1}^m  ||\nabla_g u_i||_{L^\infty(\mathbb M)}^2  \int_{\mathbb M}    |\nabla_g f_\epsilon|^2 dV_g 
\\&\le&\nonumber  \epsilon \max_{i=1}^m \{||\nabla_g u_i||_{L^\infty(\mathbb M)}^2\}  .
  \end{eqnarray} 
Sending $\epsilon\to 0$ and using the fact that the inequality (\ref{jHi}) holds,  we get the following estimate that is a counterpart of (\ref{ricM}) on $M_p$, 
\begin{equation}\label{ricMp}
\sum_{i=1}^m   \int_{M_p}  \left[  Ric_g(\nabla_g u_i,\nabla_g u_i) +|\mathcal H_{u_i}|^2-|\nabla _g |\nabla_g u_i| |^2  \right] dV_g  \le 0. 
\end{equation}
Applying similar arguments given for Part (i),  completes  the proof. 

\end{proof}

The next theorem deals with compact manifolds when $Ric_g$ is identically zero.   We show that when the manifold is compact and $Ric_g$ is precisely zero,  then stable solutions must be constant.  Comparing Theorem \ref{Mcom} with the following theorem may help us better understand the role of $Ric_g$.  We refer interested readers to \cite{j} for more information.   

\begin{thm}
Let $u=(u_i)_{i=1}^m$ be a stable solution for symmetric system (\ref{main}). Suppose that $\mathbb M$  is a compact and connected Riemannian manifold and 
 $Ric_g$ is zero.   Then,  each $u_i$ is constant for $i=1,\cdots,m$.
\end{thm}
\begin{proof} For each $i$, differentiate (\ref{main}) and multiply with $\nabla_g u_i$ to get 
\begin{equation}\label{diff1}
\nabla_g u_i \cdot  \nabla_g \Delta_g u_i  = - \sum_{j=1}^m  \partial_j H_i(u) \nabla_g u_j \cdot \nabla_g u_i.
\end{equation}
Note that (\ref{ineqfH}) implies that for each $i$
\begin{equation}\label{Hui}
|\nabla_g |\nabla_g u_i| |^2 \le |\mathcal H_{u_i}|^2. 
\end{equation}
Adding  (\ref{diff1}) and (\ref{Hui}),  we get 
  \begin{eqnarray}\label{Huigin}
|\mathcal H_{u_i}|^2 + \nabla_g u_i \cdot  \nabla_g \Delta_g u_i  \ge  |\nabla_g |\nabla_g u_i| |^2- \sum_{j=1}^m  \partial_j H_i(u) \nabla_g u_j \cdot \nabla_g u_i. 
    \end{eqnarray} 
Since $Ric_g$ is assumed to be zero,  the Bochner-Weitzenb\"{o}ck formula (\ref{bochner}) implies 
\begin{equation}\label{bochner0r}
\frac{1}{2} \Delta_g |\nabla _g u_i|^2 = |\mathcal H _{u_i}|^2+ \nabla _g \Delta_g u_i\cdot \nabla_g u_i .
\end{equation}
 Combining this and (\ref{Huigin}) we get 
\begin{equation}\label{bochner0r1}
\frac{1}{2} \Delta_g |\nabla _g u_i|^2 \ge |\nabla_g |\nabla_g u_i| |^2- \sum_{j=1}^m  \partial_j H_i(u) \nabla_g u_j \cdot \nabla_g u_i. 
\end{equation}
Integrating this on $\mathbb M$ and considering the fact that $\mathbb M$ does not have a boundary, for each $i$,  we obtain  
\begin{equation}\label{bochner0r2}
\int_{\mathbb M} \left[ |\nabla_g |\nabla_g u_i| |^2- \sum_{j=1}^m  \partial_j H_i(u) \nabla_g u_j \cdot \nabla_g u_i \right] dV_g  \le 0.
\end{equation}
Taking  summation  on index $1\le i \le m$ for (\ref{bochner0r2}), yields 
\begin{eqnarray}\label{bochner0r3}
I_1 &:=& \sum_{i=1}^m  \int_{\mathbb M}\left[ |\nabla_g |\nabla_g u_i| |^2 dV_g -  \partial_i H_i(u) |\nabla_g u_i|^2 \right]  dV_g  
\\& \le& \nonumber \sum_{i\neq j}^m  \int_{\mathbb M} \partial_j H_i(u) \nabla_g u_j \cdot \nabla_g u_i dV_g  =:I_2   .
\end{eqnarray}
On the other hand,  from the stability inequality (\ref{stability}) for symmetric systems, we have
\begin{eqnarray} \label{stability1}
I_3(\phi) &:=& \sum_{i=1}^{m} \int_{\mathbb M} \left[ |\nabla_g \phi_i|^2 - \partial_i H_i(u) \phi_i^2\right] dV_g 
\\&\ge& \nonumber \sum_{i\neq j}^m \int_{\mathbb M}  |\partial_j H_i(u)| \phi_i \phi_j  dV_g =: I_4(\phi), 
\end{eqnarray} 
when  $\phi=(\phi_i)_{i=1}^m$ is a sequence of test functions. Setting $\phi_i=|\nabla_g u_i|$ in the above $I_3(\phi)$ and $I_4(\phi)$,  we  get   
\begin{equation}
I_1=I_3(\phi) \ \ \ \text{and}\ \ \  I_2\le I_4(\phi).
\end{equation}
 This implies that $\phi=(\phi_i)_{i=1}^m$ when $\phi_i=|\nabla_g u_i|$ is the minimizer of the energy $J(\phi):=I_3(\phi)-I_4(\phi)\ge 0$.  Therefore, 
\begin{equation}
 -\Delta_g |\nabla_g u_i|= \sum_{j=1}^{n} \partial_j H_i(u) |\nabla_g u_j| \ \ \text{on } \ \mathbb M, 
 \end{equation}
for each $i=1,\cdots,m$. From the compactness of $\mathbb M$, there exists a point $\bar p_i \in \mathbb M$ such that $u_i(\bar p_i)=\max_{\mathbb M} u_i$.  This implies that $\nabla_g u_i(\bar p_i)=0$. From the strong maximum principle,  we get $|\nabla_g u_i|\equiv 0$ for each $i=1,\cdots,m$. This completes  the proof. 

\end{proof}

We now provide the following Liouville theorem in lower dimensions for any nonlinearity $H=(H_i)_{i=1}^m$ with  nonnegative components $H_i$. Note that for a complete, connected, Riemannian manifold $\mathbb M$ 
with nonnegative Ricci curvature and dimension $n$, the volume of  a ball of radius $R$, denoted $|B_R|$,  is bounded by  $R^n$.   This fact implies that the following Liouville theorem is valid for $n < 4$.  

\begin{thm} Suppose that $u=(u_i)_{i=1}^m$ is a bounded stable solution of (\ref{main}) where $H_i\ge 0$ for each $i=1,\cdots,m$.  Let  the following decay estimate hold 
\begin{equation}\label{decayB_R}
\liminf_{R\to\infty} R^{-4} |B_R| =0.
\end{equation}
Assume also that the Ricci curvature of $\mathbb M$ is nonnegative and  $Ric_g$ is not identically zero. Then,  each $u_i$ must be constant for $i=1,\cdots,m$. 
\end{thm}
\begin{proof}   Suppose that $\zeta\in C_c^\infty([-2,2])\to[0,1]$ where $\zeta\equiv 1$ on $[-1,1]$.  For  $p\in\mathbb M$ and $R>0$,  set 
\begin{equation}
\zeta_R(p)=\zeta\left( \frac{d_g(p)}{R} \right),
\end{equation}
 where $d_g$ is the geodesic distance. Therefore,   $\zeta_R\in C_c^\infty(\mathbb M)$  satisfies  $\zeta_R=1$ on $B_R$ and $\zeta_R=0$ on $\mathbb M \setminus B_{2R}$ and 
\begin{equation}
||\nabla_g \zeta_R||_{L^{\infty}(B_{2R}\setminus B_R)} \le \frac{C}{R} .  
\end{equation}   
From the assumptions,  each $u_i$ is bounded.  Multiplying both sides of system (\ref{main}) with $(u_i-||u_i||_{L^\infty (\mathbb M)})\zeta_R^2$  and  from the fact that $H_i\ge 0$,  we get    
\begin{equation}
H_i(u) [u_i-||u_i||_{L^\infty (\mathbb M)}] \zeta_R^2    \le 0 \ \ \text{in}\ \ \mathbb M. 
  \end{equation}
From this and  the equation (\ref{main}), we have 
\begin{equation}
\label{supersol}
-\hfill \Delta_g u_i (u_i-||u_i ||_{L^\infty (\mathbb M)})  \zeta_R^2  \le 0  \ \ \text{in}\ \ \mathbb M.
  \end{equation}
Doing integration by parts, for each $i=1\cdots,m$,  we obtain
\begin{equation}
\label{}
\int_{B_{2R}} |\nabla_g  u_i|^2\zeta_R^2 dV_g \le 2 \int_{B_{2R}} |\nabla_g u_i||\nabla \zeta_R |(||u_i ||_{L^\infty (\mathbb M)}-u_i)\zeta_R dV_g . 
  \end{equation}
Applying  the Cauchy-Schwarz inequality yields  
\begin{equation}\label{Bdecay}
\int_{B_{2R}}  |\nabla_g  u_i|^2 \zeta_R^2 dV_g \le C  \int_{B_{2R}} |\nabla_g \zeta_R|^2dV_g , 
  \end{equation}
where the constant  $C$ is independent from $R$. From the definition of $\zeta_R$ and applying (\ref{Bdecay}),  we get 
\begin{eqnarray} \label{ccc}
 \sum_{i=1}^m   \int_{B_{R}}  |\nabla_g  u_i|^2  dV_g &\le& \sum_{i=1}^m  \int_{B_{2R}}  |\nabla_g  u_i|^2 \zeta_R^2 dV_g
\\&\le & \nonumber C \int_{B_{2R}} |\nabla_g \zeta_R|^2dV_g 
\le  C \frac{|B_{2R}|}{R^2} .  
  \end{eqnarray}
We now apply the Poincar\'{e} inequality (\ref{poincare}) with the test function $\eta_i=\zeta_R$.  In the light of (\ref{ccc}),  the right-hand side of (\ref{poincare}) becomes
\begin{eqnarray}
 \liminf_{R\to\infty} \sum_{i=1}^m   \int_{\mathbb M}  |\nabla_g u_i|^2   |\nabla_g \eta_i|^2dV_g  &\le&   \liminf_{R\to\infty} \frac{C}{R^2} \sum_{i=1}^m \int_{B_{2R}}   |\nabla_g  u_i|^2  dV_g
\\&\le& \nonumber C  \liminf_{R\to\infty} \frac{|B_{4R}|}{R^4}=0.    
  \end{eqnarray}
Therefore, 
\begin{eqnarray}
&&  \sum_{i=1}^m   \int_{\mathbb M}  \left(  Ric_g(\nabla_g u_i,\nabla_g u_i) +|\mathcal H_{u_i}|^2-|\nabla _g |\nabla_g u_i| |^2  \right) dV_g 
\\&& +\sum_{i\neq j} \int_{\mathbb M} \left(|\partial_j H_i(u)|   |\nabla_g u_i| |\nabla_g u_j| -  \partial_j H_i(u) \nabla_g u_i \cdot\nabla_g  u_j  \right) dV_g \le 0. 
  \end{eqnarray}
  This implies that 
  \begin{equation}
  \sum_{i=1}^m   \int_{\mathbb M}  \left(  Ric_g(\nabla_g u_i,\nabla_g u_i) +|\mathcal H_{u_i}|^2-|\nabla _g |\nabla_g u_i| |^2  \right) dV_g \le 0. 
  \end{equation}
The rest of the argument is very similar to the ones provided in the proof of Theorem \ref{Mcom}. 

\end{proof}

The ideas and methods applied in the above proof are strongly motivated by the ones provided in \cite{df}, when $\mathbb M=\mathbb R^n$, and in \cite{fsv1,fsv2} when the domain is a Riemannian manifold.  Note that similar idea are used in \cite{fg} for the case of system of equations on $\mathbb R^n$. Lastly, in two dimensions, we have the following rigidity result for level sets of solutions. Note that as it is shown in \cite{dkw}, for the case of scalar equation,  the following flatness result does not hold for the Allen-Cahn equation in $\mathbb R^n$ with $n \ge 9$ endowed with its standard flat metric. 
 We assume that  $Ric_g$ is  identically zero that is equivalent to the Gaussian
curvature to be zero in two dimensions.  We refer interested readers to \cite{fv,sav,fsv1,fsv2} for similar flatness results.

\begin{thm}
Suppose that $u=(u_i)_{i=1}^m$ is a stable solution of (\ref{main}) and  each $|\nabla_g u_i|\in L^\infty(\mathbb M)$ where $\mathbb M$ is a Riemannian manifold with $\text{dim}\ \mathbb M = 2$.  Assume  that $Ric_g$ is  identically zero.  Then, every connected component of any level set of each $u_i$,  on which $|\nabla_g u_i|$  does not vanish,  must be a geodesic. 
\end{thm}
\begin{proof} Note that  assumptions imply that $\mathbb M$ is parabolic and it  has nonnegative Gaussian curvature.  Therefore,  for each $p\in \mathbb M$ there exist a set $M_p$ and a function $f_\epsilon$ such that 
\begin{equation}
\int_{\mathbb M}| \nabla_g f_\epsilon|^2 dV_g \le \epsilon .
\end{equation}
  Set $\eta_i =f_\epsilon$ in the Poincar\'{e} inequality (\ref{poincare}) to get 
 \begin{eqnarray}\label{poincare3}
&&  \sum_{i=1}^m   \int_{M_p}  \left(  Ric_g(\nabla_g u_i,\nabla_g u_i) +|\mathcal H_{u_i}|^2-|\nabla _g |\nabla_g u_i| |^2  \right) dV_g 
\\&&\nonumber +\sum_{i\neq j} \int_{M_p} \left(|\partial_j H_i(u)|   |\nabla_g u_i| |\nabla_g u_j| -  \partial_j H_i(u) \nabla_g u_i \cdot\nabla_g  u_j  \right) dV_g 
\\&\le& \nonumber \epsilon \max_{i=1}^m \{||\nabla_g u_i||_{L^\infty(\mathbb M)}^2\}  .
  \end{eqnarray} 
When $\epsilon$ approaches zero, we have 
\begin{equation}\label{ricMH1}
 \int_{ M_p}  Ric_g(\nabla_g u_i,\nabla_g u_i) dV_g =\int_{M_p} \left[ |\mathcal H_{u_i}|^2-|\nabla _g |\nabla_g u_i| |^2 \right]dV_g  =0.
 \end{equation}
Therefore, 
\begin{equation}\label{inuip1}
|\nabla_g |\nabla_g u_i| |^2 (p)= |\mathcal H_{u_i}|^2(p), 
 \end{equation} 
for every $p\in M_p\cap\{\nabla_g u_i\neq0\}$. Since the equality in (\ref{ineqfH}) holds, there exists $\kappa_k:\mathbb M\to \mathbb R$ for each $k=1,\cdots,n$ such that 
\begin{equation}\label{kappa1}
\nabla_g(\nabla_g u_i)_k(p)=\kappa_k (p) \nabla_g u_i(p).
\end{equation}
For each $i=1,\cdots,m$, consider a connected component $\Gamma$ of $\{u_i\equiv C\} \cap \{ \nabla_g u_i\neq 0\}$ that is a smooth curve.   Let $\gamma_i:\mathbb R\to\mathbb M$ with 
 \begin{equation}\label{dot}
|\dot{\gamma_i}|^2=1.
\end{equation}
It is sufficient to show that 
 \begin{equation}\label{ddot}
\ddot{\gamma_i}^k(t)+ \Gamma^k_{\lambda \mu} (p) \dot{\gamma_i}^\lambda(t) \dot{\gamma_i}^\mu(t)=0 \ \ \text{in}  \ \ \mathbb R. 
\end{equation}
We show that (\ref{ddot}) holds for an arbitrary value $\bar t\in\mathbb R$ and  $\bar p = \gamma_i(\bar t)$. Differentiating (\ref{dot}) with respect to $t$ we get 
 \begin{equation}\label{dot1}
0= \partial_k g_{\lambda \mu}(\gamma_i(t)) \dot{ \gamma_i }^k(t)  \dot{ \gamma_i}^\lambda(t)  \dot{ \gamma_i}^\mu(t) + 2  g_{\lambda \mu}(\gamma_i(t)) \dot{ \gamma_i }^k(t) \ddot{ \gamma_i}^\lambda (t). 
\end{equation}
We now use normal coordinates at some fixed point $\bar p\in \mathcal M$. Suppose that 
 \begin{equation}\label{glm}
g_{\lambda \mu}(\bar p)=\delta_{\lambda\mu}, \ \ \partial_k g_{\lambda\mu}(\bar p)=0 \ \ \text{and}\ \ \Gamma^k_{\lambda \mu}(\bar p)=0.
\end{equation}
Therefore,  
 \begin{equation}\label{dotcross}
\dot{ \gamma_i }(\bar t) \ddot{ \gamma_i} (\bar t)=0.
  \end{equation}
Note that for any point on $\Gamma$ we have $u_i(\gamma(t))\equiv C$. From this  we get 
 \begin{equation}\label{ucder1}
 0= \partial_\lambda u_i(\gamma(t)) \dot{ \gamma_i }^\lambda(t) , 
   \end{equation}
and 
\begin{equation} \label{ucder2}
 0 = \partial_{\lambda \mu} u_i(\gamma(t)) \dot{ \gamma_i}^\lambda (t) \dot{ \gamma_i }^\mu(t) +\partial_\lambda u_i(\gamma(t)) \ddot{ \gamma_i}^\lambda (t). 
  \end{equation}
Combining (\ref{ucder2})  and (\ref{kappa1}) for $p=\bar p$ and $k=\mu$, we end up with 
 \begin{equation}\label{0ee}
0= [\kappa_\mu (\bar p) \dot{ \gamma_i}^\mu (\bar t)][ \partial_\lambda u_i(\bar p) \dot{ \gamma_i}^\lambda (\bar t)] + \partial_\lambda u_i(\bar p) \ddot{ \gamma_i}^\lambda (\bar t). 
  \end{equation}
Substituting  (\ref{ucder1}) in (\ref{0ee}),  we get
\begin{equation}\label{1ee}
0= \partial_\lambda u_i(\bar p) \ddot{ \gamma_i}^\lambda (\bar t). 
  \end{equation}
From (\ref{1ee}) and (\ref{dotcross}) we conclude that $\ddot{ \gamma_i}^\lambda (\bar t)=0$, since $\ddot{ \gamma_i}^\lambda (\bar t)$  is orthogonal to $\dot{ \gamma_i }(\bar t)$, that is tangent to $\{u_i\equiv C\}$ at $\bar p$, and to $\partial_\lambda u_i(\bar p)$ that is orthogonal to $\{u_i\equiv C\}$ at $\bar p$. This and (\ref{glm}) imply that  (\ref{ddot}) holds at $\bar t$. This completes the proof. 

\end{proof}

We end this section with the fact that for the most of our main results in this section,   we assumed that the Ricci curvature is nonnegative.  We would like to refer interested readers to \cite{bm}, where hyperbolic spaces are discussed,  and references therein for rigidity results when the  curvature is negative.

\end{document}